\documentclass{amsart}
\usepackage[margin=1in]{geometry}
\usepackage{amsmath}
\usepackage{amssymb}
\usepackage{amsthm}
\usepackage{amscd}
\usepackage{enumerate}
\usepackage{xcolor}

\usepackage{graphicx}
\usepackage{amsmath,amsthm,amsfonts,amscd,amssymb,comment,eucal,latexsym,mathrsfs}
\usepackage{stmaryrd}
\usepackage[all]{xy}

\usepackage{epsfig}

\usepackage[all]{xy}
\xyoption{poly}
\usepackage{fancyhdr}
\usepackage{wrapfig}
\usepackage{epsfig}

\usepackage{hyperref}

\theoremstyle{plain}
\newtheorem{thm}{Theorem}[section]

\newtheorem{prop}[thm]{Proposition}
\newtheorem{lem}[thm]{Lemma}
\newtheorem{cor}[thm]{Corollary}

\theoremstyle{definition}

\theoremstyle{remark}
\newtheorem{remark}{Remark}

  \def\C{{\mathbb{C}}}           \def\N{{\mathbb{N}}}

 \def\cB{{\mathcal{B}}}     \def\cG{{\mathcal{G}}}      \def\cM{{\mathcal{M}}} \def\cN{{\mathcal{N}}}     \def\cS{{\mathcal{S}}}

\title{The exact quantum chromatic number of Hadamard graphs}

\author{A. Meenakshi McNamara}
\address{Perimeter Institute for Theoretical Physics,
51 Caroline Street, Waterloo, Ontario, Canada}
\address{Purdue University, Mathematical Sciences Bldg, 150 N University St, West Lafayette, IN 47907}
\email{amcnamara@perimeterinstitute.ca}
\urladdr{}

\begin{document}

\begin{abstract}
       We compute the exact value of the quantum chromatic numbers of Hadamard graphs of order $n=2^N$ for $N$ a multiple of $4$
       using the upper bound derived by Avis, Hasegawa, Kikuchi, and Sasaki, as well as an application of the Hoffman-like lower bound of Elphick and Wocjan that was generalized by Ganesan for quantum graphs. 
       As opposed to prior computations for the lower bound, our approach uses Ito's results on conjugacy class graphs allowing us to also find bounds on the quantum chromatic numbers of products of Hadamard graphs. In particular, we also compute the exact quantum chromatic number of the categorical product of Hadamard graphs.
\end{abstract}
\maketitle

\section*{Introduction}

Quantum coloring of graphs is a modification of the usual notion of graph coloring using non-local games. In a non-local game, players Alice and Bob collaborate to answer pairs of questions without communication using some prior agreed-upon strategy. The possible strategies used by the players have many classifications (see, for instance, \cite{og_qchromatic_graph} and \cite{quantum_graph_homo_op_sys}), including classical strategies and ``quantum" strategies where a shared entangled resource is available. The use of entanglement often allows quantum strategies to outperform classical ones, leading to a ``quantum advantage." In fact, the famous CHSH inequality may be described as such a quantum advantage in a non-local game. There has been especially great interest in non-local games in recent years due to their importance in quantum information theory. Further, the resolution of the Connes Embedding Conjecture through showing that MIP*=RE in \cite{ji2022mipre} relied upon non-local games. Hence, studying non-local games is important to understanding the broader impacts of this result.

In the graph coloring game, the players wish to convince a referee that they have a coloring for a graph, i.e., an assignment of colors to vertices of the graph such that adjacent vertices are different colors.
Graph coloring games are of particular interest because they provide explicit examples of non-local games, and often naturally tie into other problems in quantum information theory and communication as seen in \cite{qcolor_hadamard}. Additionally, there are numerous modifications of these games allowing for the concrete study of different models of non-local games, as has been done in \cite{quantum_graph_homo_op_sys} where the equivalence of graph coloring questions to the Connes Embedding Conjecture through Tsirelson's Problem were studied. Recent work by Harris in \cite{universality_graph_homo_games} showed that all non-local games may be reformulated as a kind of graph coloring games; furthering the importance of graph coloring in the study of non-local games.

In particular, when studying graph coloring games we are interested in the minimum number of colors needed to win. Restricting to quantum strategies, we call this the quantum chromatic number of the graph, $\chi_q$, and it is known that determining this number is an NP hard problem in general \cite{chiQNPHard}.  Numerous examples of graphs $G$ have been found exhibiting a quantum advantage such that $\chi_q(G) < \chi(G)$, where $\chi(G)$ is the usual (i.e., classical) chromatic number of $G$. In particular \cite{qcolor_hadamard} found the first such examples to be Hadamard graphs on $2^N$ vertices with $N>3$ a multiple of $4$, $H_N$. In particular for $N=12$ we have $\chi(H_{12}) = 13$ while $\chi_q(H_{12}) = 12$. These upper bounds on the quantum chromatic number of the Hadamard graph can be understood by viewing the Hadamard graph as an orthogonality graph, and similar results for general orthogonality graphs were developed in \cite{og_qchromatic_graph}. However, in these examples, the computations of the quantum chromatic of graphs give only upper bounds. Further, many of the lower bounds are given by other non-local invariants such as the  tracial rank defined in \cite{fractional}, and thus are similarly difficult to study. Hence, except for a few cases of graphs without any quantum advantage (such as the graphs studied in \cite{spectralBounds_not_priyanga} and the subfamily of Kneser graphs studied in \cite{spectralBounds2}) or with few colors needed in their classical coloring (for example, see \cite{og_qchromatic_graph}), there are very few examples of the precise quantum chromatic numbers of graphs being known.

In this paper, we will study the quantum chromatic number of the Hadamard graphs; giving precise results for an infinite family of graphs. Note that throughout this work, we consider Hadamard graphs as defined by Ito in \cite{Ito_Hadamard_graphs} and \cite{Ito_Hadamard_graphs_II} (see Sections \ref{subsec:Hadamard} and \ref{subsec:HadamardConj}). This is in contrast to another, perhaps more common, definition of Hadamard graphs that comes more directly from the notion of Hadamard matrices, such as those considered in \cite{HadamardQIso}. However, our notion of Hadamard graphs is still related to Hadamard matrices. In particular, as is discussed in \cite{Ito_Hadamard_graphs}, the conjecture of the existence of Hadamard matrices for every $N$ a multiple of $4$ can be rephrased as conjecturing that the clique number of $H_N$ is $N$. These graphs have been studied in such depth in the quantum coloring setting due to their connection to quantum information theory problems. For instance, the first results concerning quantum coloring Hadamard graphs are contained within \cite{HadamardBeforeGraphs1} for $N=2^n$, in the process of studying the problem of simulating entangled systems using classical hidden variable models and communication. In these studies, an upper bound for the quantum chromatic number of Hadamard graphs has been established in \cite{qcolor_hadamard}. In fact, using results of \cite{Hadamard_classical_coloring}, it was shown in \cite{qcolor_hadamard} that Hadamard graphs provide examples of an exponential gap between the quantum and classical chromatic numbers. These results on quantum coloring Hadamard graphs were further applied to the problem of simulating quantum entanglement via classical communication in \cite{appliedHadamard}. In this work we shall focus on lowers bounds for the quantum chromatic number.


In general, very few graphs have had concrete lower bounds on their quantum chromatic numbers established. In this work, we will make use of the spectral lower bound derived in \cite{spectralBounds_not_priyanga}. These results were generalized to quantum graphs in \cite{Priyanga}, and in fact the proof of these results was simplified using the notions of quantum coloring on quantum graphs. The lower bound for the quantum chromatic number of Hadamard graphs has been calculated in several ways. First, in \cite{graphHomoWHadamardPf} the Lov\'asz theta function was used to prove that $\chi_q(H_N) = N$. This was also proven in \cite{spectralBounds2} using spectral arguments. In the current work, we use the same spectral bound to prove the result. However, we obtain the spectrum of Hadamard graphs by viewing them as conjugacy class graphs which allows us to easily obtain further results on products of Hadamard graphs. As classical graphs are special cases of quantum graphs, we will emphasize the applicability of our approach to quantum graphs as well by phrasing results using the language of quantum graphs when convenient. This is especially relevant as interest in quantum graphs has been growing in recent years due to their connections with quantum information theory. In fact, quantum graphs were developed to generalize the role of graphs in classical information theory to the quantum setting \cite{noncomm_graph_Duan_2013}. Quantum graphs are a non-commutative topology generalization of graphs that have been independently developed several times leading to numerous equivalent descriptions (see \cite{firstQuantumGraph, noncomm_graph_Duan_2013, weaver_qgraphs_relations, comp_q_func_Musto_2018, bigalois_graphIso_2019}). The notion of quantum chromatic numbers has been extended to quantum graphs, and the quantum adjacency matrix approach of \cite{comp_q_func_Musto_2018} lends itself to spectral lower bounds as found in Ganesan's work in \cite{Priyanga} where she generalized the results of \cite{spectralBounds_not_priyanga}. In our work, these lower bounds will prove particularly useful in the study of Hadamard graphs and their products due to an equivalent description of these graphs as conjugacy class graphs, allowing for the spectral results to be found using character theory as in \cite{Ito_Hadamard_graphs}.


Using these lower bounds in addition to the more usual approach to obtaining upper bounds on quantum chromatic numbers, we obtain the main result for Hadamard graphs:
\begin{thm}[Exact quantum chromatic number of Hadamard graphs]\label{thm:main}
    Let $H_N$ be the Hadamard graph on $2^N$ vertices, $N$ a multiple of $4$. Then,
    $$\chi_q(H_N) = N.$$
\end{thm}

Graph products in relation to non-local games have also been studied in various settings \cite{graph_products} \cite{fractional} \cite{MR3537033}, and in particular bounds on quantum chromatic numbers for all four fundamental graph products of quantum graphs were studied in \cite{desantiago2024quantumchromaticnumbersproducts}. Viewing Hadamard graphs as conjugacy class graphs allows us to obtain spectral lower bounds on their graph products which may be compared with the previously obtained graph product bounds. In particular, we show the following result for categorical products
\begin{thm}[Categorical products]\label{thm:categorical}
    Let $H_N$ and $H_M$ be Hadamard graphs on $2^N$ and $2^M$ vertices respectively, with $M,N$ multiples of $4$ and $N>M$. Then,
    $$\chi_q(H_N\times H_M) = M.$$
\end{thm}

The structure of this paper is as follows. In Section \ref{section:background} we provide the necessary background on quantum coloring and quantum graphs, and the various viewpoints on Hadamard graphs. Then, Section \ref{sec:HadamardChromaticNumber} is divided into two main subsections: First, in Section \ref{subsec:upperBounds} we review the proof of the upper bound for the quantum chromatic number of Hadamard graphs, reformulating it in the language of positive operator valued measures. Then, in Section \ref{sec:lowerBounds} we review results of Ito on Hadamard graphs and develop lower bounds, leading to our result. Finally, in Section \ref{sec:Products} we consider products of Hadamard graphs.

\subsection*{Acknowledgements} The author would like to thank Rolando de Santiago for his guidance and help throughout this project. The author would also like to thank Priyanga Ganesan for thoughtful discussions about quantum coloring and non-local games. Additionally, the author would also like to thank Clive Elphick for sharing several proofs of the exact quantum chromatic number of Hadamard graphs in the literature.
Research at Perimeter Institute is supported in part by the Government of Canada through the Department of Innovation, Science and Economic Development and by the Province of Ontario through the Ministry of Colleges and Universities.

\section{Background}\label{section:background}
    \subsection{Quantum Colorings of Graphs}\label{subsec:quantumGraphsAndColorings}
    In a non-local game, cooperating players Alice and Bob play against a referee by answering questions while separated such that no communication is allowed (hence, non-local). The referee will ask both players questions, and determine if they win a given (independent) round based on their answers. Alice and Bob may agree on a strategy ahead of time, and we call a strategy winning if they win with probability 1. There are numerous strategies Alice and Bob may be allowed to use, we are primarily interested in classical (i.e. local) and quantum strategies. Other strategies often studied include quantum approximate, quantum commuting, and non-signaling.

A quantum strategy allows Alice and Bob to share an entangled state $\psi\in \C^a\otimes \C^b$ that they measure with positive operator valued measures (POVMs) $\{E_{x_A,y_A}\}\in M_a$ and $\{F_{x_B,y_B}\} \in M_b$ corresponding to the referee asking questions $x_A$ and $x_B$ to Alice and Bob, respectively, and Alice and Bob responding with $y_A$ and $y_B$ respectively. 
Using this setup, the chance that Alice and Bob respond with $y_A$ and $y_B$ given questions $x_A$ and $x_B$ is 
$$p(y_A, y_B|x_A, x_B) = \langle \psi|E_{x_A, y_A}\otimes F_{x_B, y_B}|\psi\rangle.$$ Alternatively, in a local strategy no entanglement is present. This corresponds to requiring that $\psi$ is a product state in the strategy above. Equivalently, local strategies satisfy $p(y_A, y_B|x_A, x_B) = p_1(y_A|x_A)p_2(y_B|x_B)$ for some probabilities $p_1$ and $p_2$ \cite{fractional}.

The non-local game we will consider in this paper is the graph coloring game. Classically, an $m$-coloring of a graph $G$ is an assignment of a list of colors $m$ colors to $V(G)$ such that adjacent vertices are given different colors. The chromatic number of a graph $\chi(G)$ is the smallest value of $m$ such that the graph has an $m$-coloring. 
In the \textbf{graph coloring game} Alice and Bob want to convince a referee that they have a coloring for a graph. To do so, the referee gives Alice and Bob vertices of the graph $x_A, x_B\in V(G)$, and the players respond with colors $y_A, y_B\in \{1,2,..,c\}$ for the vertices. Alice and Bob have a winning strategy if their responses follow the rules of graph coloring, i.e., if (1) $p(y_A \neq y_B|x_A=x_B) = 0$, and (2) if $p(y_A=y_B|x_A\sim x_B) = 0$. Condition (1) is precisely the definition of a synchronous game. Thus, Bob's POVMs can be obtained from Alice's as described in \cite{og_qchromatic_graph}. Classically, a winning strategy must use at least $\chi(G)$ colors. However, if Alice and Bob use a quantum strategy it may be possible to win using fewer colors. We call the number of colors needed for a winning quantum strategy the \textbf{quantum chromatic} number of a graph, $\chi_q(G)$.

These types of coloring games have also been generalized to quantum graphs in \cite{GanesanHarrisQuantumToClassical}.
One approach to studying coloring on quantum graphs by \cite{GanesanHarrisQuantumToClassical} uses the operator space notion: an (irreflexive) \textbf{quantum graph} is a tuple $(\cS, \cM, \cB(H))$ where $\cM\subset \cB(H)$ is a finite dimensional, von Neumann algebra and $\cS$ is an operator space that is a bimodule over $\cM'$ that satisfies $\cS\perp \cM'$. Here, we consider a \textbf{quantum coloring} to be a projection valued measure (PVM) $\{P_a\}_{a\in [c]}\in \cM\bar{\otimes} \cN$ where $[c]$ are the colors, $\cN$ is a finite dimensional von Neumann algebra, 
and the projections satisfy the coloring relation 
    $$P_a(S\otimes I_\cN)P_a = 0\, \forall a\in \{1,2,...,c\}.$$
We may modify the definition to generalize other strategies by changing the conditions on $\cN$ (see \cite{GanesanHarrisQuantumToClassical}).

Another approach to quantum graphs that we will use is the \textbf{quantum adjacency matrix}. We consider a quantum set to be a pair $(\cM, \psi)$ where $\cM$ is a finite dimensional C*-algebra and $\psi:\cM\to \C$ is a faithful state. Note that we may identify $\cM$ with $L^2(\cM)$, where $L^2(\cM) = L^2(\cM, \psi)$ is the GNS completion of $\cM$ with respect to $\psi$. Hence, given $m:\cM\otimes \cM\to \cM$ via multiplication we obtain $m^*$ its adjoint on $L^2(\cM)$. We say $\psi$ is a $\delta$-form if there exists $\delta>0$ such that $mm^* = \delta^2 I$.  Then, the tuple $(\cM, \psi, A)$ is a quantum graph with quantum adjacency matrix $A:L^2(\cM)\to L^2(\cM)$ if $A$ is schur idempotent ($m(A\otimes A)m^* = \delta^2 A$) and satisfies $(I\otimes \eta^*m)(I\otimes A\otimes I)(m^*\eta\otimes I) = A$ where $\eta:\C\to \cM$ via $\lambda \mapsto \lambda I$. An irreflexive quantum graph further satisfies $m(A\otimes I)m^* = 0$. Throughout the remainder of this paper we will assume that all (quantum) graphs are irreflexive as coloring ideas are best applied in this setting. These two equivalent definitions are described and compared in \cite{Priyanga}.

The quantum adjacency matrix approach has allowed many spectral lower bounds to be extended to quantum coloring as was done in \cite{Priyanga}. In particular, we will make use of the bound \textbf{Hoffman bound} for quantum coloring of classical graphs that was found in \cite{spectralBounds_not_priyanga}, phrased here for its generalization to quantum graphs given by \cite{Priyanga}: 
Let $\cG = (\cM, \psi, A)$ be a quantum graph, and let $\lambda_{max} = \lambda_1\geq \lambda_2\geq\cdots\geq \lambda_{\dim \cM} = \lambda_{min}$ be all the eigenvalues of $A$ (with multiplicity). Then
    $1 + \frac{\lambda_{max}}{|\lambda_{min}|}\leq \chi_q(\cG).$


Both notions of quantum graphs generalize classical graphs, and hence results for quantum graphs specialize to classical graphs. In the operator space approach a classical graph corresponds to a quantum graph where $\cM$ is abelian. Likewise, in the quantum adjacency matrix approach a classical graph corresponds to the usual notion of adjacency matrices, i.e., a $|V|\times |V|$ matrix with $1$ at $ij$th entry if $i\sim j$ in $G$ and zeros elsewhere.

    \subsection{Hadamard Graphs as Orthogonality Graphs}\label{subsec:Hadamard}

    Let $N\in \N$ be a multiple of $4$. The Hadamard graph of size $n=2^N$ is a graph defined as follows: the vertices are vectors of size $N$ with entries $\pm 1$; two vertices are adjacent iff the inner product of the vectors is $0$.  We will denote the vertex $x$ as the vector $(x_i)_{i=0}^{N-1}$ where $x_i$ is the $i$-th component of the vector. In this notation, the vertices $x$ and $y$ are adjacent iff the vectors $(x_i)_{i=0}^{N-1}$ and $(y_i)_{i=0}^{N-1}$ differ in exactly half of the entries. 
Note that this makes the Hadamard graph an orthogonality graph as we are representing the vertices as vectors and edges are determined by the orthogonality relation.

    \begin{remark}\label{rmrk:Hadamard_connected}
    Note that the Hadamard graph defined above is disconnected. In particular, vertices with an even number of $+1$s and $-1$s and those with an odd number form identical connected components. Hence, we may restrict the Hadamard graph to the graph on $2^{N-1}$ vertices correspondong to the connected component containing the all $1$s vector (the component with even numbers of $+1$s and $-1$s).
    \end{remark}

    \begin{remark}\label{rmrk:Hadamard_as_0s_1s}
        It is  convenient to define the POVMs for coloring using a slightly different notation for the Hadamard graph. Namely, the Hadamard graph of size $n=2^N$ is defined as follows: the vertices are vectors of size $N$ with entries $0$s and $1$s; two vertices are adjacent iff exactly half of the vector entries differ.
        Note that this is equivalent to the above definition with $-1$s replaced by $0$s. This convention will be applied for notational convenience.
    \end{remark}

    \subsection{Hadamard Graphs as Conjugacy Class Graphs}\label{subsec:HadamardConj}

    In \cite{Ito_Hadamard_graphs} it was shown that the Hadamard graph can also be viewed as a conjugacy class graph in the sense of \cite{Ito_conjugacy_class_graphs}. In a conjugacy class graph, the vertices may be identified with elements of a group $\Gamma$. Further, there is some $C\subset \Gamma$ that is a union of conjugacy classes, satisfies $C = C^{-1}$ and generates $\Gamma$ such that vertices $x$ and $y$ are adjacent iff $x = cy$ for some $c\in C$.  Characterizing the Hadamard graph as a conjugacy class graph will allow for the use of irreducible characters in order to calculate the spectrum of the adjacency matrix of the Hadamard graph, as will be further developed in Section \ref{sec:lowerBounds}.
    
    To view the Hadamard graph as a conjugacy class graph we will consider the Hadamard graph as described in Remark \ref{rmrk:Hadamard_connected}. Following \cite{Ito_Hadamard_graphs}, we see that the vertices form an abelian group with group multiplication given by componentwise multiplication of the vectors and unit the all $1$s vector $h$. If we let $C$ be the set of vertices adjacent to $h$, then vertices $x$ and $y$ are adjacent iff $x = cy$ for some $c\in C$. We note that $C$ is the union of conjugacy classes since in an abelian group conjugacy classes are single elements. Additionally, $C = C^{-1}$ as each element is its own inverse. To see that $C$ generates the group, let $w(i,j)$ be the vector of all $+1$s except at the $i$th and $j$th spot where there are $-1$s. Observe $\{w(0,1),w(0,2),\ldots,w(0,N)\}$ is a generating set for the group. Further, the elements $w(0,i_1)\cdot w(0,i_2)\cdots w(0,i_{N/2})$ and $w(0,i_2)~\cdot~ w(0,i_3)~\cdots ~w(0,i_{N/2})$, for $0 < i_k < N$ and $i_j\neq i_k$ for all $j,k$, are both in $C$, and we notice that $\binom{N-1}{N/2} + \binom{N-1}{N/2 -1} = \binom{N}{N/2}$ so all elements of $C$ are of this form. Then we notice that the product of these elements is $w(0, i_1)$ for any $i_1\in \{0,\ldots,N-1\}$, so $C$ generates the group.

\section{Computing the Chromatic Number}\label{sec:HadamardChromaticNumber}
    \subsection{Upper Bounds}\label{subsec:upperBounds}

    In this section, we will describe the work of \cite{qcolor_hadamard} which shows that the quantum chromatic number of the Hadamard graph on $2^N$ vertices with $N$ a multiple of $4$ is less than or equal to $N$. Note that, as is noted in \cite{qcolor_hadamard}, this result along with \cite{Hadamard_classical_coloring} implies an exponential gap between the chromatic number and quantum chromatic number of Hadamard graphs with $N=4p^q$. Hence, this work provided an early example of quantum advantage in non-local games.

In the work of \cite{qcolor_hadamard}, the authors described a winning strategy to color the Hadamard graph on $2^N$ vertices using $N$ colors. Given vertices $(a_i)$ and $(b_i)$ in the convention used in Remark \ref{rmrk:Hadamard_as_0s_1s}, Alice and Bob win using the following protocol:
\begin{enumerate}
    \item Prepare an initial state $|\psi\rangle := \frac{1}{\sqrt{N}}\sum_{j=0}^{N-1}|j\rangle \otimes |j\rangle$
    \item Alice and Bob apply phase shifts using $P_{a_i}$ and $P_{b_i}$, where $P_{l_i}|i\rangle = (-1)^{l_i}|i\rangle$. This results in $(P_{a_i}\otimes P_{b_i})|\psi\rangle$.
    \item Apply the general quantum Fourier transform to get 
    $$(QFT_N\otimes QFT_N^{-1})(P_{a_i}\otimes P_{b_i})|\psi\rangle = \frac{1}{N^{3/2}}\sum_{j=0}^{N-1}\sum_{k=0}^{N-1}\sum_{\ell=0}^{N-1}(\omega)^{\ell(j-k)}(-1)^{a_i\oplus b_i}|j\rangle\otimes|k\rangle.$$
    \item Measure the state using this computational basis. This will obtain one of $N$ basis states, and they return the corresponding color to the referee.
\end{enumerate}

We wish to reframe this strategy into the familiar POVM framework of section \ref{subsec:quantumGraphsAndColorings}. Recall that as this is a synchronous game we may assume that Alice and Bob use the same POVM, and also that they are projections. Hence, we must find a set of projections $\{E_{x,\alpha}\}_{x\in V(H_n), \alpha \in [n]}$ such that $\sum_{\alpha\in [n]} E_{x,\alpha} = I$, and such that using these projections and their conjugates for Alice and Bob is a winning strategy (namely, we must check that given adjacent vertices they return different colors).

\begin{thm}[{Theorem 3.1 in \cite{qcolor_hadamard}}]\label{thm:Hadamard_upperbd}
    Let $G$ be the Hadamard graph on $2^{N}$ vertices where $N$ is a multiple of 4.  Then
    $\chi_q(G)\leq N$.
\end{thm}

\begin{proof}

We will adopt the convention for Hadamard graphs used in Remark \ref{rmrk:Hadamard_as_0s_1s}.
Fix $\omega$, a primitive root of unity. For each $\alpha\in [N]$ $x\in V(G)$,  define $$\langle \widetilde{E}_x^\alpha| = \frac{1}{\sqrt{N}}\sum_{j=0}^{N-1}\omega^{j\alpha}(-1)^{x_j}\langle j|$$
    and set $E_x^\alpha =| \widetilde{E}_x^a\rangle \langle \widetilde{E}_x^\alpha| $. Note that since  $\widetilde{E}_x^a$ is a unit vector, it is readily checked that $E_x^\alpha$ is a projection for each $x\in V(G)$ and $\alpha\in[N] $. 
    Consider the shared entangled state $|\psi \rangle =\frac{1}{\sqrt{N}} \sum_{j=0}^{N-1}|j\rangle \otimes |j\rangle$
We now verify that $\{E_x^\alpha\}_{\alpha\in [N]}$ forms a PVM for each $x\in V(G)$.
To this end, 
\begin{align*}
    \sum_{\alpha = 1}^N E_x^\alpha &= \sum_{\alpha=1}^N \frac{1}{N}\sum_{j,k=0}^{N-1}\omega^{(j-k)\alpha}(-1)^{x_j + x_k}|k\rangle \langle j|\\
    &=\left( \frac{1}{N}\sum_{j,k=0}^{N-1}(-1)^{x_j + x_k}|k\rangle \langle j|\right)\sum_{\alpha=1}^N \omega^{(j-k)\alpha}\\
    &= \sum_{j=0}^{N-1} |j\rangle \langle j| = I
\end{align*}
where in the second to the last line, we used the following facts: 
 if $j=k$ then $(-1)^{x_j + x_k} = 1$ and $\sum_{\alpha=1}^N \omega^{0} = N$; and if $j\neq k$ then we get $\sum_{\alpha=1}^N \omega^{(j-k)\alpha} = 0$.

Next, we verify the coloring condition. First we calculate,
\begin{align*}
    \langle \psi| E_x^\alpha \otimes \overline{E_y^\beta}|\psi\rangle =& \frac{1}{N^3}\sum_{\ell=0}^{N-1}\langle \ell|\otimes \langle \ell|\left(\sum_{j,k=0}^{N-1} \omega^{(j-k)\alpha}(-1)^{x_j + x_k }|k\rangle \langle j|\right) \otimes \\ &\left(\sum_{j',k'=0}^{N-1} \omega^{(k'-j')\beta}(-1)^{y_{j'} + y_{k'} }|k'\rangle \langle j'|\right)\sum_{\ell=0}^{N-1}| \ell'\rangle \otimes | \ell'\rangle\\
    =& \frac{1}{N^3}\sum_{\ell,\ell'=0}^{N-1}\left( \omega^{(\ell'-\ell)\alpha}(-1)^{x_{\ell'} + x_\ell }\right)\left(\omega^{(\ell-\ell')\beta}(-1)^{y_{\ell'} + y_{\ell} }\right)\\
        =& \frac{1}{N^3}\sum_{\ell, \ell'=0}^{N-1} \omega^{(\ell' - \ell)(\alpha - \beta)}(-1)^{x_\ell + x_{\ell'} + y_{\ell} + y_{\ell}'}
\end{align*}
Thus we have
\begin{equation}\label{eq:PVMComputation}
     \langle \psi| E_x^\alpha \otimes \overline{E_y^\beta}|\psi\rangle =\frac{1}{N^3}\sum_{\ell = 0}^{N-1}\omega^{\ell'(\alpha - \beta)}(-1)^{x_{\ell'} + y_{\ell'}}\sum_{\ell'=0}^{N-1}\omega^{\ell(\beta-\alpha)}(-1)^{x_\ell + y_\ell}
\end{equation}
Now we wish to verify that if $x\sim y$ and $\alpha=\beta$ then $ \langle \psi| E_x^\alpha \otimes \overline{E_y^\beta}|\psi\rangle $ evaluates to $0$. Using these assumptions in Equation \eqref{eq:PVMComputation}, we find
\begin{align*}
  \langle \psi| E_x^\alpha \otimes \overline{E_y^\beta}|\psi\rangle =&= \frac{1}{N^3}\sum_{\ell = 0}^{N-1}(-1)^{x_\ell + y_\ell}\sum_{\ell'=0}^{N-1}(-1)^{x_{\ell'} + y_{\ell'}} = 0
\end{align*}
since exactly half of the $x_\ell + y_\ell$ are even and exactly half are odd when $x\sim y$.

Finally, since Bob's POVMs are found from Alice's POVMs as described in \cite{og_qchromatic_graph} this must be a synchronous game. Hence, it is easily checked that if $\alpha\neq \beta$ then
\begin{align*}
    \langle \psi| E_x^\alpha \otimes \overline{E_x^\beta}|\psi\rangle = 0.
\end{align*}
\end{proof}

\begin{remark}
    
We obtain the PVM corresponding to colors used in coloring quantum graphs as in \cite{GanesanHarrisQuantumToClassical} by considering the following block diagonal matrix, $$P_\alpha = \begin{pmatrix}
    E_1^\alpha & 0 &\cdots  &0 \\
    0& E_2^\alpha &\cdots &0 \\
    \vdots &\vdots& \ddots&\vdots \\
   0 &0 &\cdots &  E_n^\alpha   
\end{pmatrix},$$ where the vertices are labeled using the set $[n]$. Notice that we have $P_\alpha \in D_n\otimes M_N(\C)$ as $E_x^\alpha \in M_N(\C)$, so as $M_N(\C)$ is finite dimensional this is a quantum coloring in the above description as expected. 

The verification that $\{P_\alpha\}_\alpha$ forms a PVM is is immediate from the result for $\{E_x^\alpha\}_\alpha$, and the coloring condition may be checked a quick calculation verifying that
$$P_\alpha (E_{x,y}\otimes I_N)P_\alpha = 0$$
whenever $x\sim y$, where $E_{x,y}\in M_n(\C)$ is the elementary matrix with a $1$ at the $xy$ entry and zero elsewhere.
\end{remark}

\subsection{Lower Bounds}\label{sec:lowerBounds}

    We will now derive a lower bound for the quantum chromatic number of the Hadamard graph. In particular, we will use the Hoffman bound that was shown in \cite{Priyanga} using eigenvalues of the quantum adjacency matrix. As this is a classical graph, the quantum adjacency matrix is precisely the usual concept of an adjacency matrix. Hence, we will use the results of Ito in \cite{Ito_conjugacy_class_graphs} and \cite{Ito_Hadamard_graphs} in order to compute the eigenvalues by viewing the Hadamard graph as a conjugacy class graph.

    We begin by noting that conjugacy class graphs are determined by a group and a (union of) conjugacy classes of the group. Hence, representation theory is useful in studying such graphs, and we shall review some of the relevant definitions. In particular, a representation of a group $\Gamma$ on a vector space $V$ is a group homomorphism $\rho:\Gamma \to GL(V)$. A representation is \textit{irreducible} if there are no invariant subspaces of $\rho(\Gamma)$ except for $\{0\}$ and $V$ itself. If $V$ is finite dimensional (as will be considered here), then the \textit{character} corresponding to $\rho$ is $\chi_\rho:\Gamma\to \C$ via $\chi_{\rho}(g) = Tr(\rho(g))$ for any $g\in \Gamma$. A character $\chi_\rho$ is called irreducible if $\rho$ is irreducible. 
    Characters possess several properties that will be used in the following:
    \begin{itemize}
        \item Characters are constant on conjugacy classes of $\Gamma$ since $Tr(ab)=Tr(ba)\implies Tr(aba^{-1}) = Tr(b)$.

        \item The set of class functions (functions that take a constant value on conjugacy classes) has a natural inner product:$            \langle \alpha, \beta\rangle = \frac{1}{|\Gamma|}\sum_{g\in \Gamma}\alpha(g)\overline{\beta(g)}.$
         Further, if $\alpha$ and $\beta$ are characters corresponding to a finite dimensional representation, and $\Gamma$ is of finite order then
        \begin{align}\label{eqn:innerproduct}
            \langle \alpha, \beta\rangle = \frac{1}{|\Gamma|}\sum_{g\in \Gamma}\alpha(g)\beta(g^{-1}),
        \end{align}
        because finite order invertible linear transformations have the trace of their inverse is the complex conjugate of the trace.

        \item The set of irreducible characters forms an orthonormal basis for the class functions. In particular, if $\rho = \rho_1\oplus \rho_2$ is a reducible representation, then $\chi_\rho = \chi_{\rho_1} + \chi_{\rho_2}$.

        \item For $h\in \Gamma$ and $\Gamma$ of finite order, we will also use 
        \begin{align}\label{eqn:diffCharRelation}
            \sum_{g\in \Gamma}\chi_s(g^{-1})\chi_r(gh) = |\Gamma|\frac{\chi_r(h)}{\chi_r(e)}\delta_{r,s}
        \end{align}
        for any two irreducible characters $\chi_r$ and $\chi_s$.

        \item An important example of a representation is the regular representation, defined by $R(g):\ell^2(\Gamma)\to \ell^2(\Gamma)$ via $R(g)e_h = e_{gh}$ for all $g,h\in \Gamma$ and $\{e_g\}$ the basis for $\ell^2(\Gamma)$. Note that $\ell^2(\Gamma) \cong \C^{|\Gamma|}$ when $\Gamma$ is of finite order. Then if we let $\gamma$ be the character corresponding to $R$, we find $\gamma(g) = |\Gamma|\delta_{g,e}$, since $R(g)$ is a permutation matrix with zero trace unless $g=e$. Further, $\langle \gamma, \chi\rangle = \frac{1}{|\Gamma|}\sum_{g\in \Gamma}\gamma(g)\overline{\chi(g)} = \frac{1}{|\Gamma|}\gamma(e)\chi(e) = \chi(e)$, so if $\{\chi_i\}_{i=1}^k$ is the set of irreducible representations of $\Gamma$ then \begin{align}\label{eqn:regrep_irreps}
            \gamma = \sum_{i=1}^k \chi_i(e)\chi_i.
        \end{align}
        Hence, we find $|\Gamma| = \sum_{i=1}^k\chi_i(e)^2$
    \end{itemize}

    The proof of the following theorem is contained within \cite{Ito_conjugacy_class_graphs}. We include it for completeness.
    \begin{thm}[\cite{Ito_conjugacy_class_graphs}]\label{thm:spectrum_Ito}
        Suppose $G$ is a conjugacy class graph defined using group $\Gamma$ and conjugacy class (or inverse closed union of conjugacy classes) $C$. Further, suppose that $\chi$ is an irreducible character of $\Gamma$. Then $\chi$ contributes to the spectrum of $G$ an eigenvalue $\lambda = |C|\chi(c)/\chi(e)$ (or $\lambda = \sum_{c\in C}\chi(c)/\chi(e)$), where $c\in C$, with multiplicity $\chi(e)^2$. (Note that distinct characters may add to the multiplicity of a single eigenvalue).
    \end{thm}
    \begin{proof} 

    In this proof we shall assume that $C$ is a single conjugacy class for ease of notation. If $C$ is an inverse closed union of conjugacy classes then the proof follows identically but with $|C|$ replaced by $\sum_{c\in C}$ throughout.
    
        Let $V(G) = \{x_1,\ldots,x_g\}$ with $g = |\Gamma|$ as vertices are identified with elements of $\Gamma$. Further, let $\delta_C(z) = 1$ or $0$ depending on if $z\in C$ or not. Then we see that the adjacency matrix of $G$ is $A = (\delta_C(x_ix_j^{-1}))_{i,j\leq 1}^g.$
        
        Since $\delta_C$ is a class function it is a $\C$ linear combination of irreducible characters of $\Gamma$, i.e.~$\delta_C = \sum_{i = 1}^k a_i\chi_i,$ where $a_i\in \C$ and $\chi_i$, $1\leq i\leq k$, are the irreducible characters. Then orthogonality of irreducible characters using the inner products in Equation \eqref{eqn:innerproduct} gives $$g\langle \delta_C, \chi_i\rangle = a_ig = \sum_{j=1}^g\delta_C(x_j)\chi_i(x_j^{-1}) = \sum_{c\in C} \chi_i(c^{-1}) = |C|\chi_i(c)$$ for $c\in C$. Hence, $a_i = |C|\chi_i(c)/g$.

        Now, letting $D_\ell = (\chi_\ell(x_ix_j^{-1}))_{i,j=1}^g$ we get $A = \frac{|C|}{g}\sum_{\ell=1}^k\chi_\ell(c)D_\ell.$ Additionally, let $X_{s,1} = \left(\chi_s(x_1^{-1}),\ldots,\chi_s(x_g^{-1})\right)$ for $1\leq s\leq k$. Then $X_{s,1}A =  \frac{|C|}{g} \sum_{\ell=1}^k \chi_\ell(c)X_{s,1}D_\ell$.
        
        The relation of group characters given in Equation \eqref{eqn:diffCharRelation} gives us the $j$th component of $X_{s,1}D_\ell$ equals $$\sum_{i=1}^g\chi_s(x_i^{-1})\chi_\ell(x_ix_j^{-1}) = \left(\chi_s(x_j^{-1})/\chi_s(e)\right)g\delta_{s,\ell}.$$ Hence, the $j$th component of $X_{s,1}A$ is $|C|\left(\chi_s(c)/\chi_s(e)\right)\chi_s(x_j^{-1})$, and we see that $X_{s,1}$ is an eigenvector of $A$ with eigenvalue $|C|\chi_s(c)/\chi_s(e).$

        We may further notice that $X_{s,1}$ is the first row vector of $D_s$, so let $X_{s,m}$ be the $m$th row vector of $D_s$. Similarly, we find these are eigenvectors of $A$ with eigenvalues $|C|\chi_s(c)/\chi_s(e)$ as well. 
        Additionally, we have 
        $$X_{s,l}\cdot X_{t,m} = \sum_{i=1}^g \chi_s(x_\ell x_i^{-1})\chi_t(x_i x_m^{-1}) = \sum_{i=1}^g \chi_s(x_i^{-1})\chi_t(x_ix_\ell x_m^{-1}) = 0$$ if $s\neq t$. Hence, we complete the proof if we can show that $D_s$ has rank $\chi_s(e)^2$ for $1\leq s\leq k$.

        Let $R$ be the regular representation of $\Gamma$, and $\gamma$ the character of $R$. 
        Using Equation \eqref{eqn:regrep_irreps} we get $(\gamma(x_ix_j^{-1})) = \sum_{\ell=1}^k \chi_\ell(e)D_\ell$. If we let $\{e_i\}_{i=1}^g$ be the standard basis vectors for $\C^g$, then this gives $ge_i = \sum_{\ell=1}^k\chi_\ell(e)X_{i,\ell}$. Hence, $\{X_{s,i}\}_{s,i=1}^{s=k, i=g}$ generates $\C^g$, and since $g = \sum_{\ell=1}^k \chi_\ell(e)^2$ it suffices to show that the rank of $D_\ell$ does not exceed $\chi_\ell(e)^2$.

        Let $R_\ell(x) = (a_{rs}^\ell(x))_{r,s=1}^{\chi_\ell(e)}$ for $x\in \Gamma$ be an irreducible representation of $\Gamma$ corresponding to the character $\chi_\ell$. Further, let $$A_i^{(\ell)} = \begin{pmatrix}
            a_{i1}^\ell(x_1) & \cdots & a_{i\chi_\ell(e)}^\ell(x_1)\\
            \vdots & & \vdots\\
            a_{i1}^\ell(x_g) & \vdots & a_{i\chi_\ell(e)}^\ell(x_g)
        \end{pmatrix}\begin{pmatrix}
            a_{1i}^\ell(x_1^{-1}) & \cdots & a_{1i}^\ell(x_g^{-1})\\
            \vdots & & \vdots\\
            a_{\chi_\ell(e)i}^\ell(x_1^{-1}) & \vdots & a_{\chi_\ell(e)}^\ell(x_g^{-1})
        \end{pmatrix}$$
        for $1\leq i\leq \chi_\ell(e)$. We have $\chi_\ell(x_jx_i^{-1}) = \sum_{r,t=1}^{\chi_\ell(e)}a_{rt}^\ell(x_i)a_{tr}^\ell(x_j^{-1})$. Hence, we find $D_\ell = A_1^{(\ell}) +\cdots + A_{\chi_\ell(e)}^{(\ell)}$. Since the rank of $A_i^{(\ell)}$ cannot exceed $\chi_\ell(e)$ we get the rank of $D_\ell$ cannot exceed $\chi_\ell(e)^2$.
    \end{proof}

        Applying the above theorem to the Hadamard graph, we obtain the following result. Note that this result is contained in \cite{Ito_Hadamard_graphs} Proposition 5. However, we restrict our attention to only the largest and smallest eigenvalues, as only these are needed to apply the Hoffman bound. The proof of the following is also found in \cite{Ito_Hadamard_graphs}. We include a sketch of the relevant portions of the proof for completeness.
        
    \begin{thm}[\cite{Ito_Hadamard_graphs}, Part of Proposition 5]\label{thm:Hadamard_spectrum}
    Let $A$ be the adjacency matrix for the Hadamard graph on $2^N$ vertices where $N$ is a multiple of 4. Then the largest eigenvalue of $A$ is $\lambda_{max} = \binom{N}{N/2}$ and the smallest eigenvalue is $\lambda_{min} = -\frac{\lambda_{max}}{N-1}.$
    \end{thm}
    \begin{proof}
        As was shown in Section \ref{subsec:Hadamard}, we may view the Hadamard graph as a conjugacy class graph for an abelian group $\Gamma$ and picking the union of conjugacy classes to be the set of vertices adjacent to the all $1$s vertex $h$, $D_1(h)$. Hence, we need only understand the characters of $\Gamma$ and their values on the $D_1(h)$ by Theorem \ref{thm:spectrum_Ito}. Further, as $\Gamma$ is abelian the characters are precisely $1$-dimensional representations, and for any character $\chi$,  $\chi(e) = 1$.

        Recall that we denote the vector of all ones except at the $i$th and $j$th spots by $w(i,j)$, and that the set $\{w(0,1),\ldots,w(0,N-1)\}$ forms a generating set for $\Gamma$. Then define the character $\chi_{0,i}$, $1\leq i\leq N-1$, by $\chi_{0,i}(w(0,i)) = -1$ and $\chi_{0,1}(w(0,j)) = 1$ for $i\neq j$. Then $\{\chi_{0,1},\ldots,\chi_{0,N-1}\}$ forms a generating system for the character group of $\Gamma$. 
        
        Let $\chi_0$ be the identity character. Then $\chi_0(c) = 1$ for any $c\in D_1(h)$, so we get the eigenvalue associated to $\chi_0$ is $|D_1(h)| = \binom{N}{N/2}$. This is the maximal eigenvalue as it is clear from the generating set of characters  that no character will evaluate to value larger than $1$.

        It can be shown that any product of $r$ distinct characters in the generating set, say $\chi = \chi_{0,i_1}\cdots\chi_{0,i_r}$, will contribute $0$ to the spectrum of $A$ if $r$ is odd. Thus, we now the case when
        $2\leq r\leq N$ even. Let the product of $r$ distinct characters in the generating set be $\chi = \chi_{0,i_1}\cdots\chi_{0,i_r}$. This character contributes the eigenvalue
        \begin{align*}
          \lambda(r) =& \sum_{v\in D_1(h)} \chi(v)\\
          =& \binom{N}{N/2}-2\sum_{i \text{ odd}}\binom{r}{i}\binom{n-r}{(N/2)-i}\\
          =& a_{N,r}(N/2) = (-1)^{r/2}\frac{r!(N-r)!}{(N/2)!(r/2)!((N-r)/2)!}
        \end{align*}
        where $a_{N,r}(N/2)$ is the coefficient of $x^{N/2}$ in $(1-x)^r(1+x)^{N-r}$, and the last equality is a result of K.~Nomura as described in \cite{Ito_conjugacy_class_graphs}. Moreover $\lambda(r+4)/\lambda(r) = (r+3)(r+1)/(n-r-1)(n-r-3)$, so a quick calculations shows that picking $r=N-2$ gives $\lambda(r) = -\frac{\lambda(N)}{N-1}$ is the minimum eigenvalue.
    \end{proof}

    Then combining these spectral results for Hadamard graphs in Theorem \ref{thm:Hadamard_spectrum} with the Hoffman lower bound from \cite{spectralBounds_not_priyanga} and \cite{Priyanga} gives the main result.  
    \begin{cor}\label{cor:main}
        For $N$ a multiple of $4$, let $H_N$ be the Hadamard graph on $2^N$ vertices. Then  
        $\chi_q(H_N) = N.$
    \end{cor}

\section{Products of Hadamard graphs}\label{sec:Products}

There are four standard graph products. Namely, the categorical, Cartesian, and strong products which are all commutative, as well as the lexicographic product which is not. For graphs $G$ and $H$, each of these products forms a graph on the vertex set $V(G)\times V(H)$ as follows:
\begin{itemize}
    \item The \textbf{categorical product} $G\times H$ has $(v,w)\sim (x,y)\in V(G)\times V(H)$ if and only if $v\sim x\in G$ and $w\sim y$.
    \item The \textbf{Cartesian product} $G\square H$ $(v,w)\sim(x,y)\in V(G)\times V(H)$ if and only if $v=x$ and $w\sim y$ or $v\sim x$ and $w=y$.

    \item The \textbf{strong product} $G\boxtimes H$ has $(v,w)\sim(x,y)\in V(G)\times V(H)$ if and only if $v\sim x$ and $w=y$, $v=x$ and $w\sim y$, or $v\sim x$ and $w\sim y$.

    \item The \textbf{lexicographic product} $G[H]$ has $(v,w)\sim(x,y)\in V(G)\times V(H)$ if and only if $v\sim x$ or $v=x$ and $w\sim y$.
\end{itemize}

We may consider graph products of conjugacy class graphs $G$ and $H$ corresponding to groups $\Gamma_G$ and $\Gamma_H$ and (unions of) conjugacy classes $C_G\subset \Gamma_G$ and $C_H\subset \Gamma_H$. We immediately see that any of the above graph products may be viewed as conjugacy class graphs on the group formed by the direct product $\Gamma_G\times \Gamma_H$. Noting that the identity is a conjugacy class on its own, the following lemma describing the resultant conjugacy class for each product is immediate.

\begin{lem}\label{lemma:productsDef}
Let $G$ and $H$ be conjugacy class graphs as described above with conjugacy classes $C_G$ and $C_H$. Then the following describes their products.
\begin{itemize}
    \item $G\times H$ has conjugacy class $C_G\times C_H$.
    \item $G\square H$ has conjugacy class $({e_G}\times C_H) \sqcup (C_G\times {e_H})$.
    \item $G\boxtimes H$ has conjugacy class $({e_G}\cup C_G)\times ({e_H}\cup C_H) \setminus \{(e_G,e_H)\}$ where we remove the identity since there are no self loops.
    \item $G[H]$ has conjugacy class $(C_G\times \Gamma_H)\cup ({e_G}\times C_H)$.
\end{itemize}
\end{lem}

We note that the irreducible representations of $\Gamma_G\times \Gamma_H$ are of the form $\rho_G\otimes \rho_H$ where $\rho_G$ and $\rho_H$ are irreducible representations of $\Gamma_G$ and $\Gamma_H$. Hence, the irreducible characters of $\Gamma_G\times \Gamma_H$ are of the form $\chi = \chi_G\chi_H$ where $\chi_G$ and $\chi_H$ are irreducible characters of $\Gamma_G$ and $\Gamma_H$. Given such an irreducible character, Ito's Theorem \ref{thm:spectrum_Ito} associates $\lambda = \sum_{(c,d)\in C} \chi_G(c)\chi_H(d)/\chi(e_G)\chi(e_H) = \lambda_G\lambda_H$ where $C$ is the conjugacy class for the product, $\lambda_G$ is associated to $\chi_G$ and $\lambda_H$ is associated to $\lambda_H$. We will use these results to give bounds on the quantum chromatic numbers of products of conjugacy class graphs, and apply these results to Hadamard graphs. Throughout this section let $G$ and $H$ be conjugacy class graphs as above corresponding to conjugacy classes and groups $C_G\subset \Gamma_G$ and $C_H\subset \Gamma_H$. Further, let $\lambda_{G,min},\lambda_{G,max}, \lambda_{H,min},\lambda_{H,max}$ be the minimum and maximum eigenvalues of $G$ and $H$.

Considering the categorical product we find the following proposition.
\begin{prop}\label{prop:conjCat}
    Take $G$ and $H$ non-empty graphs and their eigenvalues as above. Assume WLOG that $\lambda_{G,min}\lambda_{H,max}\leq \lambda_{H,min}\lambda_{G,max}$. Then,
    $$\chi_q(G\times H)\geq  1 + \frac{\lambda_{H,max}}{|\lambda_{H,min}|}.$$
\end{prop}
\begin{proof}
    First, since the trace of an adjacency matrix is zero any non-empty graph must have both positive and negative eigenvalues. Hence, since the conjugacy class of $G\times H$ is $C = C_G\times C_H$, we find the eigenvalue corresponding to 
    $\chi = \chi_G \chi_H$, where $\chi_G,\chi_H$ are irreducible characters of $G$ and $H$, is $\lambda = \sum_{c\in C_G,d \in C_H}\chi_G(c) \chi_H(d)/ \chi_G(e_G) \chi_H(e_H) = \lambda_G\lambda_H$ 
    where $\lambda_G$ and $\lambda_H$ are in the spectrum of $G$ and $H$ respectively. Hence, the maximum eigenvalue of the spectrum of $G\times H$ is $\lambda_{max} = \lambda_{G,max}\lambda_{H,max}$, and the minimum eigenvalue is $\lambda_{min} = \min\{\lambda_{G,min}\lambda_{H,max}, \lambda_{G,max}\lambda_{H,min}\}$. Thus, the Hoffman bound gives us $\chi_q(G\times H) \geq \frac{\lambda_{max}}{|\lambda_{min}|}+1 = 1 + \frac{\lambda_{H,max}}{|\lambda_{H,min}|}$.
\end{proof}

\begin{remark}
    Note that the above bound is exactly the Hoffman bound for $H$. 
\end{remark}

Applying Proposition \ref{prop:conjCat} to Hadamard graphs along with Theorem 4.4 of \cite{desantiago2024quantumchromaticnumbersproducts} to obtain an upper bound $\chi_q(G\times H)\leq \min\{\chi_q(G), \chi_q(H)\}$, we find the following exact result.
\begin{cor}
    If $H_N$ and $H_M$ are Hadamard graphs on $2^N$ and $2^M$ vertices with $N,M$ multiples of $4$ and $N\geq M$, then 
    $$\chi_q(H_N\times H_M) = M.$$
\end{cor}
\begin{remark}
    Take $H\sqcup H$ to be the disjoint union of two copies of a graph $H$. Then the categorical, Cartesian, and strong products of graphs $G$ and $H\sqcup H$ are the disjoint union of two copies of the product of $G$ and $H$. Thus, we may calculate the quantum chromatic number of products of Hadamard graphs using only the fully connected component of the graphs (i.e., the conjugacy class graph component). 
\end{remark}

Now consider the Cartesian graph product. In this case the original conjugacy class is not the product of the conjugacy classes. Hence, we find, 
\begin{prop}\label{prop:conjCart}
    Take $G$ and $H$ non-empty conjugacy class graphs with eigenvalues as above. Then
    $$\chi_q(G\square H) \geq 1 + \frac{\lambda_{G,max} + \lambda_{H,max}}{|\lambda_{G,min} + \lambda_{H,min}|}.$$
\end{prop}
\begin{proof}
    This proof is analogous to the proof of Proposition \ref{prop:conjCat}. The only difference comes from noting that the eigenvalue corresponding to the irreducible character $\chi = \chi_G\chi_H$ is 
    \begin{align*}
        \sum_{(c,d)\in ({e_G}\times C_H) \sqcup (C_G\times {e_H})}\chi_G(c)\chi_H(d)/\chi_G(e_G)\chi_H(e_H) =& \sum_{c\in C_G}\chi_G(c)/\chi_G(e_G) + \sum_{d\in C_H}\chi_H(d)/\chi_H(e_H) \\=& \lambda_G + \lambda_H,
    \end{align*}
    where $\lambda_G$ and $\lambda_H$ are in the spectrum of $G$ and $H$ respectively. Hence, we find a sum
    rather than a product of eigenvalues of the original graphs.
\end{proof}

\begin{remark}
    Applying this bound to Hadamard graphs we find a lower bound. However, it does not appear to improve upon the bound $\chi_q(G\square H) \geq \max\{\chi_q(G), \chi_q(H)\}$ derived in \cite{desantiago2024quantumchromaticnumbersproducts}.
\end{remark}

Next consider the strong product. As this product contains all edges in both the categorical and Cartesian products we may prove the following proposition in an analogous manner to Propositions \ref{prop:conjCat} and \ref{prop:conjCart} by noting the eigenvalues associated to $\chi = \chi_G\chi_H$ are $\lambda_G\lambda_H + \lambda_G + \lambda_H$.
\begin{prop}\label{prop:conjStrong}
    Take $G$ and $H$ non-empty conjugacy class graphs with eigenvalues as above. Then, assuming WLOG that $\lambda_{G,min}\lambda_{H,max}\leq \lambda_{H,min}\lambda_{G,max}$.
    $$\chi_q(G\square H) \geq 1 + \frac{\lambda_{G,max}\lambda_{H,max} + \lambda_{G,max} + \lambda_{H,max}}{|\lambda_{G,min}\lambda_{H,max} + \lambda_{G,min} + \lambda_{H,min}|}.$$
\end{prop}
\begin{remark}
     Again, the lower bound obtained for Hadamard graphs does not appear improve upon the bound found in \cite{desantiago2024quantumchromaticnumbersproducts}.
\end{remark}

Finally, let us consider the lexicographic product. 
\begin{prop}\label{prop:conjLex}
    Take $G$ and $H$ non-empty conjugacy class graphs with eigenvalues as above. Then,
    $$\chi_q(G[H]) \geq 1 + \frac{\max\{\lambda_{G,max}|\Gamma_H| + |C_H|, \lambda_{H,max}\}}{|\min\{\lambda_{G,min}|\Gamma_G| + |C_H|, \lambda_{H,min}\}|}.$$
\end{prop}
\begin{proof}
    The eigenvalue associated to the irreducible character $\chi = \chi_G\chi_H$, for $\chi_G$ and $\chi_H$ irreducible characters of $\Gamma_G$ and $\Gamma_H$ respectively, is 
    \begin{align*}
        \lambda =& \sum_{c\in C_G, d\in \Gamma_H}\chi_G(c)\chi_H(d)/\chi_G(e_G)\chi_H(e_H) + \sum_{h\in C_H}\chi_H(h)/\chi_H(e_H)\\
        =& \frac{\lambda_G}{\chi_H(e_H)}\sum_{d\in \Gamma_H}\chi_H(d) + \lambda_H.
    \end{align*}
    
    We now use an important property of characters: if $\alpha$ is a character of a group $\Gamma$, then $$\sum_{g\in \Gamma}\alpha(g) = \begin{cases}
        |\Gamma| & \text{ if } \alpha \text{ is trivial,}\\
        0 & \text{ else}.
    \end{cases}$$ Additionally, if $\chi_H$ is the trivial character then the associated eigenvalue in $H$ is $\lambda_H = |C_H|$ .
    
    Hence, we find $$\lambda = \begin{cases}
        |\Gamma_H|\lambda_G + |C_H| & \text{ if } \chi_H \text{ is trivial},\\
        \lambda_H & \text{else}.
    \end{cases}$$
    The remainder of the proof is analogous to the proof of Proposition \ref{prop:conjCat}.
\end{proof}

Applying this to Hadamard graphs we find,
\begin{cor}
    If $H_N$ and $H_M$ are Hadamard graphs on $2^N$ and $2^M$ vertices with $N,M$ multiples of $4$ and $M>1$, then 
    $$\chi_q(H_N[H_M]) \geq N+1.$$
\end{cor}
\begin{proof}
    The conjugacy class $C_M$ for $H_M$ is all vertices adjacent to the all $1$s vector $h$. Thus, $|C_M| = \binom{M}{M/2}$. Additionally, it is clear that the quantum chromatic number of $G[H]$ is the same as $G[H\sqcup H]$ where $H\sqcup H$ is the disjoint union of two copies of $H$. Hence, we may compute the quantum chromatic number using the conjugacy class graph view of $H_N$ and $H_M$ on $2^{N/2}$ and $2^{M/2}$ vertices. The group for $H_M$ is $\Gamma_M$, and has $2^{M/2}$ elements. Using Proposition \ref{prop:conjLex}, we now find that the maximum eigenvalue of $H_N[H_M]$ is $\lambda_{max} = \max\left\{\binom{N}{N/2}2^{M/2} + \binom{M}{M/2}, \binom{M}{M/2}\right\} = \binom{N}{N/2}2^{M/2} + \binom{M}{M/2}$. Similarly, the minimum eigenvalue is $\lambda_{min} = \min\left\{- \frac{\binom{N}{N/2}2^{M/2}}{N-1} + \binom{M}{M/2}, -\frac{\binom{M}{M/2}}{M-1}\right\}$. 

    From this, we see that if $\lambda_{min} = - \frac{\binom{N}{N/2}2^{M/2}}{N-1} + \binom{M}{M/2}$, then 
    \begin{align*}
        \chi_q(H_N[H_M]) \geq& 1 + \frac{ \binom{N}{N/2}2^{M/2} + \binom{M}{M/2}}{\frac{\binom{N}{N/2}2^{M/2}}{N-1} - \binom{M}{M/2}}
        \geq 1 + \frac{\binom{N}{N/2}2^{M/2} + \binom{M}{M/2}}{\frac{\binom{N}{N/2}2^{M/2}}{N-1}}\\
        =& 1 + (N-1)\left(1 + \frac{\binom{M}{M/2}}{\binom{N}{N/2}}2^{-M/2}\right) > N.
    \end{align*}

    Similarly, if $\lambda_{min} = -\frac{\binom{M}{M/2}}{M-1}$ then 
    \begin{align*}
        \chi_q(H_N[H_M])\geq&1 + \frac{ \binom{N}{N/2}2^{M/2} + \binom{M}{M/2}}{\binom{M}{M/2}/(M-1)}\\
        =& 1 + (M-1)\left(1 + 2^{M/2}\frac{\binom{N}{N/2}}{\binom{M}{M/2}}\right) \geq M + (M-1)\binom{N}{N/2}>N.
    \end{align*}
    Thus, in either case we always have $\chi_q(H_N[H_M]) \geq N+1$. 
\end{proof}
While this is a seemingly weak lower bound on the quantum chromatic number, only upper bounds were previously known for lexicographic products (aside from bounds that apply to all quantum chromatic numbers). Further, while it is possible that the upper bound found in \cite{desantiago2024quantumchromaticnumbersproducts} is exact and hence also a lower bound, this bound relies on computing the $b$-fold quantum chromatic number of a graph which is itself a non-local invariant. In contrast, this spectral result gives an immediately calculable lower bound and further shows that the quantum chromatic number must increase for products of Hadamard graphs.

\bibliographystyle{amsalpha}
\bibliography{references.bib}

\newcommand{\etalchar}[1]{$^{#1}$}
\providecommand{\bysame}{\leavevmode\hbox to3em{\hrulefill}\thinspace}
\providecommand{\MR}{\relax\ifhmode\unskip\space\fi MR }
\providecommand{\MRhref}[2]{%
  \href{http://www.ams.org/mathscinet-getitem?mr=#1}{#2}
}
\providecommand{\href}[2]{#2}
\begin{thebibliography}{CMN{\etalchar{+}}07}

\bibitem[AHKS06]{qcolor_hadamard}
David Avis, Jun Hasegawa, Yosuke Kikuchi, and Yuuya Sasaki, \emph{A quantum
  protocol to win the graph colouring game on all hadamard graphs}, IEICE
  Trans. Fundam. Electron. Commun. Comput. Sci. \textbf{E89-A} (2006), no.~5,
  1378–1381.

\bibitem[BCE{\etalchar{+}}20]{bigalois_graphIso_2019}
Michael Brannan, Alexandru Chirvasitu, Kari Eifler, Samuel Harris, Vern
  Paulsen, Xiaoyu Su, and Mateusz Wasilewski, \emph{Bigalois extensions and the
  graph isomorphism game}, Comm. Math. Phys. \textbf{375} (2020), no.~3,
  1777--1809. \MR{4091496}

\bibitem[BCT99]{HadamardBeforeGraphs1}
Gilles Brassard, Richard Cleve, and Alain Tapp, \emph{Cost of exactly
  simulating quantum entanglement with classical communication}, Phys. Rev.
  Lett. \textbf{83} (1999), 1874--1877.

\bibitem[BGH22]{GanesanHarrisQuantumToClassical}
Michael Brannan, Priyanga Ganesan, and Samuel~J. Harris, \emph{The
  quantum-to-classical graph homomorphism game}, J. Math. Phys. \textbf{63}
  (2022), no.~11, Paper No. 112204, 34. \MR{4505911}

\bibitem[CM24]{HadamardQIso}
Ada Chan and William~J. Martin, \emph{Quantum isomorphism of graphs from
  association schemes}, Journal of Combinatorial Theory, Series B \textbf{164}
  (2024), 340--363.

\bibitem[CMN{\etalchar{+}}07]{og_qchromatic_graph}
Peter~J. Cameron, Ashley Montanaro, Michael~W. Newman, Simone Severini, and
  Andreas Winter, \emph{On the quantum chromatic number of a graph}, Electron.
  J. Combin. \textbf{14} (2007), no.~1, Research Paper 81, 15. \MR{2365980}

\bibitem[dSM24]{desantiago2024quantumchromaticnumbersproducts}
Rolando de~Santiago and A.~Meenakshi McNamara, \emph{Quantum chromatic numbers
  of products of quantum graphs}, 2024.

\bibitem[DSW13]{noncomm_graph_Duan_2013}
Runyao Duan, Simone Severini, and Andreas Winter, \emph{Zero-error
  communication via quantum channels, noncommutative graphs, and a quantum
  {L}ov\'asz number}, IEEE Trans. Inform. Theory \textbf{59} (2013), no.~2,
  1164--1174. \MR{3015725}

\bibitem[EKS98]{firstQuantumGraph}
J.~A. Erdos, A.~Katavolos, and V.~S. Shulman, \emph{Rank one subspaces of
  bimodules over maximal abelian selfadjoint algebras}, J. Funct. Anal.
  \textbf{157} (1998), no.~2, 554--587. \MR{1638277}

\bibitem[EW19]{spectralBounds_not_priyanga}
Clive Elphick and Pawel Wocjan, \emph{Spectral lower bounds for the quantum
  chromatic number of a graph}, Journal of Combinatorial Theory, Series A
  \textbf{168} (2019), 338--347.

\bibitem[FR87]{Hadamard_classical_coloring}
Peter Frankl and Vojt\v~ech R\"odl, \emph{Forbidden intersections}, Trans.
  Amer. Math. Soc. \textbf{300} (1987), no.~1, 259--286. \MR{871675}

\bibitem[Gan23]{Priyanga}
Priyanga Ganesan, \emph{Spectral bounds for the quantum chromatic number of
  quantum graphs}, 2023, pp.~351--376. \MR{4603835}

\bibitem[GRvSS16]{MR3537033}
Chris Godsil, David~E. Roberson, Robert \v~S\'amal, and Simone Severini,
  \emph{Sabidussi versus {H}edetniemi for three variations of the chromatic
  number}, Combinatorica \textbf{36} (2016), no.~4, 395--415. \MR{3537033}

\bibitem[GTW13]{appliedHadamard}
Viktor Galliard, Alain Tapp, and Stefan Wolf, \emph{Deterministic quantum
  non-locality and graph colorings}, Theoretical Computer Science \textbf{486}
  (2013), 20--26, Theory of Quantum Communication Complexity and Non-locality.

\bibitem[Har24]{universality_graph_homo_games}
Samuel Harris, \emph{Universality of graph homomorphism games and the quantum
  coloring problem}, Annales Henri Poincaré (2024).

\bibitem[Ito84]{Ito_conjugacy_class_graphs}
Noboru Ito, \emph{The spectrum of a conjugacy class graph of a finite group},
  Math. J. Okayama Univ. \textbf{26} (1984), 1--10. \MR{779767}

\bibitem[Ito85a]{Ito_Hadamard_graphs}
\bysame, \emph{Hadamard graphs. {I}}, Graphs Combin. \textbf{1} (1985), no.~1,
  57--64. \MR{796183}

\bibitem[Ito85b]{Ito_Hadamard_graphs_II}
\bysame, \emph{Hadamard graphs. {II}}, Graphs Combin. \textbf{1} (1985), no.~4,
  331--337. \MR{951024}

\bibitem[Ji13]{chiQNPHard}
Zhengfeng Ji, \emph{Binary constraint system games and locally commutative
  reductions}, ArXiv \textbf{abs/1310.3794} (2013).

\bibitem[JNV{\etalchar{+}}22]{ji2022mipre}
Zhengfeng Ji, Anand Natarajan, Thomas Vidick, John Wright, and Henry Yuen,
  \emph{Mip*=re}, 2022.

\bibitem[KM19]{graph_products}
Se-Jin Kim and Arthur Mehta, \emph{Chromatic numbers, {S}abidussi's theorem and
  {H}edetniemi's conjecture for non-commutative graphs}, Linear Algebra Appl.
  \textbf{582} (2019), 291--309. \MR{3992428}

\bibitem[McR14]{graphHomoWHadamardPf}
Laura Man\v~cinska and David Roberson, \emph{Graph homomorphisms for quantum
  players}, 9th {C}onference on the {T}heory of {Q}uantum {C}omputation,
  {C}ommunication and {C}ryptography, LIPIcs. Leibniz Int. Proc. Inform.,
  vol.~27, Schloss Dagstuhl. Leibniz-Zent. Inform., Wadern, 2014, pp.~212--216.
  \MR{3354695}

\bibitem[MRV18]{comp_q_func_Musto_2018}
Benjamin Musto, David Reutter, and Dominic Verdon, \emph{A compositional
  approach to quantum functions}, J. Math. Phys. \textbf{59} (2018), no.~8,
  081706, 42. \MR{3849575}

\bibitem[OP16]{quantum_graph_homo_op_sys}
Carlos~M. Ortiz and Vern~I. Paulsen, \emph{Quantum graph homomorphisms via
  operator systems}, Linear Algebra Appl. \textbf{497} (2016), 23--43.
  \MR{3466632}

\bibitem[PSS{\etalchar{+}}16]{fractional}
Vern~I. Paulsen, Simone Severini, Daniel Stahlke, Ivan~G. Todorov, and Andreas
  Winter, \emph{Estimating quantum chromatic numbers}, J. Funct. Anal.
  \textbf{270} (2016), no.~6, 2188--2222. \MR{3460238}

\bibitem[Wea21]{weaver_qgraphs_relations}
Nik Weaver, \emph{Quantum graphs as quantum relations}, J. Geom. Anal.
  \textbf{31} (2021), no.~9, 9090--9112. \MR{4302212}

\bibitem[WED20]{spectralBounds2}
Pawel Wocjan, Clive Elphick, and Parisa Darbari, \emph{Spectral lower bounds
  for the quantum chromatic number of a graph---{P}art {II}}, Electron. J.
  Combin. \textbf{27} (2020), no.~4, Paper No. 4.47, 11. \MR{4245222}

\end{thebibliography}

\end{document}